\documentclass[12pt]{amsart}
\usepackage[dvips]{graphicx}
\theoremstyle{plain}
\newtheorem{thm}{Theorem}
\newtheorem{lem}[thm]{Lemma}

\newtheorem{prop}[thm]{Proposition}
\theoremstyle{remark}
\newtheorem{rem}[thm]{Remark}
\begin{document}
\title[Regular homotopic deformation]{Regular homotopic deformation of compact surface with boundary and mapping class group}
\author{Susumu Hirose}
\address{Department of Mathematics, 
Faculty of Science and Engineering, 
Saga University, 
Saga, 840--8502 Japan}
\email{hirose@ms.saga-u.ac.jp}
\author{Akira Yasuhara}
\address{Department of Mathematics, 
Tokyo Gakugei University, 
Nukuikita 4--1--1, 
Koganei, Tokyo 184--8501 Japan}
\email{yasuhara@u-gakugei.ac.jp}
\thanks{This research was supported by Grant-in-Aid for 
Scientific Research (C) (No. 20540083 and No. 20540065), 
Japan Society for the Promotion of Science. } 
\begin{abstract}
A necessary and sufficient algebraic condition for 
a diffeomorphism over a surface embedded in $S^3$ to be induced 
by a regular homotopic deformation is discussed, and 
a formula for the number of signed pass moves needed for 
this regular homotopy is given. 
\end{abstract}
\maketitle
\baselineskip=20pt
\section{Introduction}
\begin{figure}[ht]
\begin{center}
\includegraphics[height=5cm]{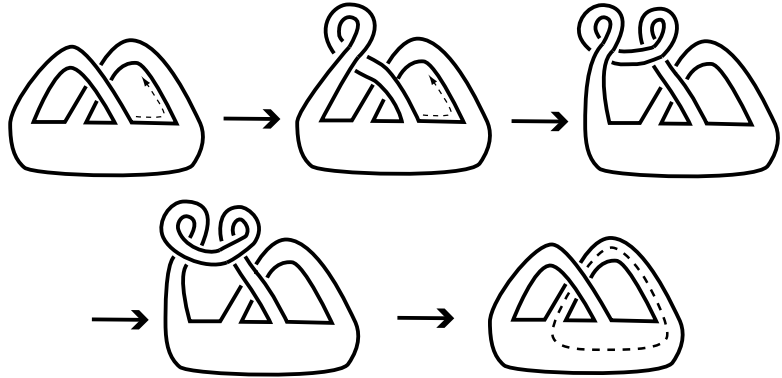}
\end{center}
\caption{An example of regular homotopic deformation of a torus with one boundary. 
The result is a double Dehn-twist about the dashed circle in the last surface. }
\label{fig:deformation}
\end{figure}
Two embeddings $f_1$ and $f_2$ of a surface $F$ into $S^3$ are called 
{\em regular homotopic\/} $f_1 \sim f_2$, if there is a homotopy $h_t \, (0 \leq t \leq 1)$ 
such that $h_0 = f_1$, $h_1 = f_2$, and $h_t$ is 
an immersion of $F$ into $S^3$ for each $t \in [0,1]$. 
Figure \ref{fig:deformation} is an example of regular homotopy. 
This regular homotopy induces a double Dehn-twist about the dashed circle 
drawn in the last surface. 
That is to say, some regular homotopic deformation induces 
an isotopically non-trivial diffeomorphism on the surface. 
In \S 2, we gives a necessary and sufficient algebraic condition for 
a diffeomorphism over the surface to be induced 
by a regular homotopic deformation. 
In this section, we observe that, in order to perform 
the above regular homotopy, pass-moves are enough. 
In \S 3, we give a formula which gives a number of 
pass-moves we need. 
\section{An algebraic characterization of a map induced by a regular homotopic deformation}
Let $F$ be a compact oriented surface with boundary, and let $f$ be a smooth embedding of 
$F$ into $S^3$. 
Let $\mathcal{M}_F$ be the mapping class group of $F$, namely 
$\mathcal{M}_F = \pi_0 ( \{ \phi \in \mathrm{Diff}_+ (F)\, ; \, 
\phi|_{\partial F} = id_{\partial F}\})$. 
In this paper, we mainly investigate on the subgroup of $\mathcal{M}_F$ 
defined by 
$$\mathcal{E}_f = \left\{ \phi \in \mathcal{M}_F\, ; \, f \circ \phi 
\text{ is regular homotopic to } f \right\}.$$
We remark here that, in the above regular homotopy, the boundary of $F$ can be moved. 
We define a quadratic form $q_f : H_1(F; \mathbb{Z}_2) \to \mathbb{Z}_2$ associated 
the embedding $f$ by $q_f (x) = lk (f_*(x), f_*(x)^+) \mod 2$, 
where $f_*(x)^+$ is $f_*(x)$ pushed into the positive normal direction of $f(F)$. 
Let $\mathcal{O}_{q_f}$ be the subgroup of $\mathcal{M}_F$ defined by 
$\mathcal{O}_{q_f}= \{ \phi \in \mathcal{M}_F \, ; \, q_f ( \phi_*(x) ) = q_f (x) 
\text{ for all } x \in H_1(F; \mathbb{Z}_2) \}$. 
Then we have the following theorem. 
\begin{thm}\label{thm:MCG}
$\mathcal{E}_f = \mathcal{O}_{q_f}$. 
\end{thm}
\begin{rem}
Nowik \cite{Nowik} proved the same type theorem for closed surface. 
\end{rem}
In order to proof this theorem, we will use the following fact. 
Although this fact is already proven by Pinkall \cite{Pinkall}, 
we will prove this by an elementary method. 
\begin{lem}\label{lem:form}
Embeddings $f$ and $f'$ of $F$ into $S^3$ are regular homotopic if and only if 
$q_f = q_{f'}$. 
\end{lem}
\begin{proof}
At first, we show that 
if $f$ and $f'$ are regular homotopic then $q_f = q_{f'}$. 
For any $x \in H_1(F; \mathbb{Z}_2)$, we take a simple closed curve $\xi$ on $F$ 
which represents $x$ and let $N(\xi)$ be a regular neighborhood of $\xi$ in $F$. 
Since $f$ and $f'$ are regular homotopic, $f|_{N(\xi)}$ and $f'|_{N(\xi)}$ 
are regular homotopic. 
Let $f(\xi)^+$ (resp. $f'(\xi)^+$) be a simple closed curve which is 
$f(\xi)$ (resp. $f'(\xi)$) pushed into the positive normal direction of $f(N(\xi))$ 
(resp. $f'(N(\xi))$), 
then $lk( f(\xi), f(\xi)^+) = lk( f'(\xi), f'(\xi)^+) \mod 2$ by \cite{Kato}. 
Therefore, $q_f (x) = q_{f'} (x)$. 

\begin{figure}[ht]
\begin{center}
\includegraphics[height=4.5cm]{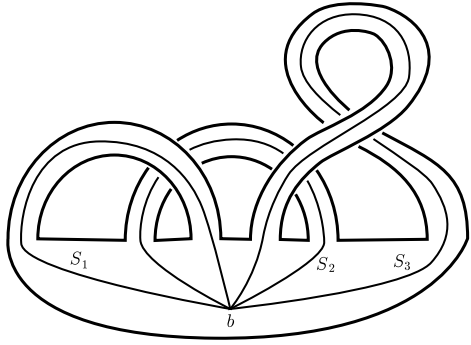}
\end{center}
\caption{An example of spine}
\label{fig:spine}
\end{figure}
Next, we prove that if $q_f = q_{f'}$ then $f$ and $f'$ are regular homotopic. 
Let $g$ be the genus of $F$ and $k$ be the number of connected components of 
$\partial F$. 
Then, as drawn in Figure \ref{fig:spine}, 
we take a bouquet $S$ in $F$ which is constructed from $2g+k-1$ simple closed curves 
$S_1, \ldots, S_{2g+k-1}$ by gluing in one point $b$, such that 
$F \setminus S$ is $k$ annuli. 
This $S$ is called a {\em spine\/} of $F$. 
Let $N(S_i)$ be a thin regular neighborhood of $S_i$ in $F$ and $N(S)$ be the union of 
$N(S_1), \ldots, N(S_{2g+k-1})$, then $N(S)$ is a regular neighborhood of $S$ in $F$. 
Let $r_S$ be an embedding from $F$ to itself such that $r_S(F) = N(S)$, $r_S|_S = id$, 
and $r_S$ is isotopic to the identity. 
By a regular homotopy, we deform $f$ and $f'$ such that $f(b) = f'(b) = b$. 
Since $q_f = q_{f'}$, $q_f([S_i]) = q_{f'}([S_i])$ for each $i$, in other word, 
$lk (f(S_i), f(S_i)^+) = lk( f'(S_i), f'(S_i)^+) \mod 2$ for each $i$. 
By this condition on linking numbers, 
we can construct a regular homotopy between $f|_{N(S_i)}$ and $f'|_{N(S_i)}$ 
which is a composition of {\em pass moves\/} (see Figure \ref{fig:pass-move}) 
and isotopies as follows. 
Since pass moves are crossing changes for the knot $f(S_i)$ in $S^3$, 
we can change $f(S_i)$ into $f'(S_i)$ by some pass moves and isotopies. 
Let $f_1$ be the embedding of $N(S_i)$ regular homotopic to $f|_{N(S_i)}$ 
which is obtained as a result of above pass moves and isotopies. 
Then, by the previous condition on linking numbers, 
the difference between the number of twists of $f_1(N(S_i))$ and 
that of $f'(N(S_i))$ 
is an even integer. 
Since this difference is dissolved by pass moves 
(see Figure \ref{fig:double-twist}), 
we can construct regular homotopy between $f_1$ and $f'|_{N(S_i)}$, 
and as a result, between $f|_{N(S_i)}$ and $f'|_{N(S_i)}$. 
Since $N(S)$ is a union of $N(S_i)$'s, we can construct a regular homotopy 
between $f |_{N(S)}$ and $f'|_{N(S)}$. 
From the above construction, we can obtain a sequence of regular homotopies, 
$f \sim f \circ r_S = (f |_{N(S)}) \circ r_S  \sim 
(f'|_{N(S)}) \circ r_S = f' \circ r_S \sim f'$. 
Therefore, $f$ is regular homotopic to $f'$. 
\end{proof}
\begin{figure}[ht]
\begin{center}
\includegraphics[height=2.5cm]{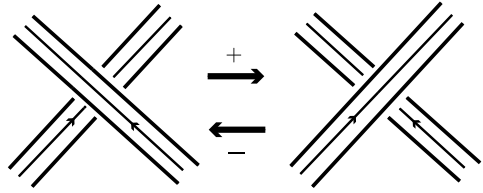}
\end{center}
\caption{Left to right is $+$ pass move and right to left is $-$ pass move }
\label{fig:pass-move}
\end{figure}
\begin{figure}[ht]
\begin{center}
\includegraphics[height=4.5cm]{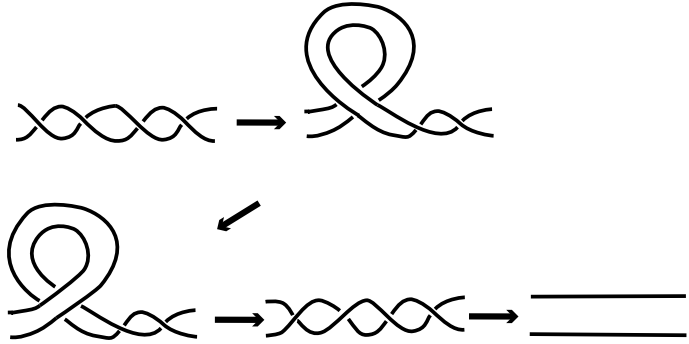}
\end{center}
\caption{A pass move is used in the second move}
\label{fig:double-twist}
\end{figure}
Now, we prove Theorem \ref{thm:MCG}. 
\begin{proof}[Proof of Theorem \ref{thm:MCG}]
We remark that $q_f \circ \phi_* = q_{f \circ \phi}$. 
If $\phi \in \mathcal{O}_{q_f}$ then $q_{f \circ \phi} = q_f \circ \phi_* = q_f$. 
By Lemma \ref{lem:form}, $f \circ \phi$ is regular homotopic to $f$, 
in other word, $\phi \in \mathcal{E}_f$. 
On the other hand, if $\phi \in \mathcal{E}_f$, then $f \circ \phi$ is 
regular homotopic to $f$. 
By Lemma \ref{lem:form}, $q_{f \circ \phi} = q_f$, that is to say, 
$q_f \circ \phi_* = q_f$. 
\end{proof}
Furthermore, we see that $\mathcal{E}_f$ completely determines the regular homotopy class 
of $f$. 
\begin{prop}
$f$ is regular homotopic to $f'$ if and only if $\mathcal{E}_f = \mathcal{E}_{f'}$
\end{prop}
\begin{proof}
We assume that $f$ is regular homotopic to $f'$. 
If $\phi \in \mathcal{E}_f$, then $f \circ \phi \sim f$. 
Therefore, $f' \circ \phi \sim f \circ \phi \sim f \sim f'$, hence, 
$\phi \in \mathcal{E}_f'$. Similarly, we can see 
that if $\phi \in \mathcal{E}_f'$ then $\phi \in \mathcal{E}_f$. 

On the other hand, we assume that $f \not\sim f'$. 
Then, by Lemma \ref{lem:form}, $q_f \not= q_{f'}$. 
Therefore, there is an element $\alpha$ of $H_1 (F ; \mathbb{Z}_2)$ 
such that $q_f(\alpha) \not= q_{f'}(\alpha)$, 
say $q_f(\alpha)=1$, $q_{f'}(\alpha)=0$. 
Let $a$ be a simple closed curve on $F$ which represents $\alpha$, 
then $T_a \in \mathcal{O}_{q_f}$ but $T_a \not\in \mathcal{O}_{q_{f'}}$. 
By Theorem \ref{thm:MCG}, $T_a \in \mathcal{E}_f$ but $T_a \not\in \mathcal{E}_{f'}$. 
Therefore, $\mathcal{E}_f \not= \mathcal{E}_{f'}$. 
\end{proof}
\section{A formula of a signed number of pass moves}
We remark that, in the proof of Lemma \ref{lem:form}, 
we use only pass moves and isotopies as a regular homotopy between 
$f$ and $f'$. 
We call a regular homotopy which is a composition of pass moves and isotopies 
{\em pass homotopy\/}. 
In the following, we investigate on some formula which gives a number of 
``signed" pass moves. 

As in the proof of Lemma \ref{lem:form}, 
we choose a spine $S = S_1 \vee \cdots \vee S_{2g+k-1}$ of $F$, 
and orient the simple closed curves $S_i$. 
A pass move which changes a negative crossing into a positive crossing 
is called a $+$ pass move, and 
a pass move which changes a positive crossing into a negative crossing 
is called a $-$ pass move (see Figure \ref{fig:pass-move}). 
We choose a basis $\{\alpha_i \,; 1 \leq i \leq 2g+k-1 \}$ of $H_1(F; \mathbb{Z})$ 
by $\alpha_i = [S_i]$, and determine a quadratic form 
$\tilde{q}_f : H_1(F; \mathbb{Z}) \to \mathbb{Z}$ by 
$\tilde{q}_f(x) = lk(f_*(x), f_*(x)^+)$. 

\begin{thm}\label{thm:signed-number}
Let $\phi$ be an element of $\mathcal{E}_f$. 
For any pass homotopy from $f$ to  $f \circ \phi$, 
$$
\begin{aligned}
&(\text{ the number of $+$ pass moves }) - (\text{ the number of $-$ pass moves })
 \\
&=\frac{1}{2} 
\left\{ \tilde{q}_f ((\phi)_*(\alpha_1 + \alpha_2+ \cdots + \alpha_{2g+k-1})) 
- \tilde{q}_f (\alpha_1 + \alpha_2+ \cdots + \alpha_{2g+k-1}) \right\}. 
\end{aligned}
$$
\end{thm}

\begin{proof}
By Theorem \ref{thm:MCG}, 
$q_f ( \phi_*(x) ) = q_f (x)$ for all $x \in H_1(F; \mathbb{Z}_2)$, 
hence,   
$\tilde{q}_f ((\phi)_*(\alpha_1 + \alpha_2+ \cdots + \alpha_{2g+k-1})) 
- \tilde{q}_f (\alpha_1 + \alpha_2+ \cdots + \alpha_{2g+k-1})$ is an even integer. 
Since a linking form is bilinear, 
$$
\begin{aligned}
\frac{1}{2} 
& \{ \tilde{q}_f ((\phi)_*(\alpha_1 + \alpha_2+ \cdots + \alpha_{2g+k-1})) 
- \tilde{q}_f (\alpha_1 + \alpha_2+ \cdots + \alpha_{2g+k-1}) \} \\
& = 
\frac{1}{2} 
\left\{ \sum_{i,j} lk \left( (f\circ \phi)_*(\alpha_i), \, 
((f\circ \phi)_*(\alpha_j))^+ \right) 
- lk \left( f_*(\alpha_i), \, (f_*(\alpha_j))^+ \right) \right\}. 
\end{aligned}
$$
Here, we set 
$$
\Delta_{i,j} = lk \left( (f\circ \phi)_*(\alpha_i), \, 
((f\circ \phi)_*(\alpha_j))^+ \right) 
- lk \left( f_*(\alpha_i), \, (f_*(\alpha_j))^+ \right). 
$$

\begin{figure}[ht]
\begin{center}
\includegraphics[height=2.5cm]{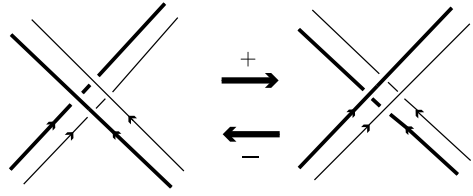}
\end{center}
\caption{Thick line is a curve in the spine, and thin line is a curve in the spine 
pushed off to the positive normal direction}
\label{fig:Delta}
\end{figure}

We assume that $f\circ \phi$ is obtained from $f$ are regular homotpic 
by one pass move. 
We determine the value of $\Delta_{i,j}$ according to the cases on the pass move
(see Figure \ref{fig:Delta}). 
If the pass move is the $+$ pass move (resp. $-$ pass move) among 
$S_{i_0}$ and itself, then 
$$
\Delta_{i,j} = 
\begin{cases}
+2 \quad (\text{resp.} -2) & \qquad i=j=i_0 \\
0 \quad (\text{resp.} 0) & \qquad otherwise. 
\end{cases} 
$$
If the pass move is the $+$ pass move (resp. $-$ pass move) among 
$S_{i_0}$ and $S_{j_0}$ $(i_0 \not= j_0)$, then 
$$
\Delta_{i,j} = 
\begin{cases}
+1 \quad (\text{resp.} -1) & \qquad (i,j)=(i_0,j_0) \\
+1 \quad (\text{resp.} -1) & \qquad (i,j)=(j_0,i_0) \\
0 \quad (\text{resp.} 0) & \qquad otherwise. 
\end{cases} 
$$
In each cases, 
for the $+$ pass move, $\sum_{i,j} \Delta_{i,j} = +2$, 
for the $-$ pass move, $\sum_{i,j} \Delta_{i,j} = -2$. 

In general, there are several pass moves in a pass homotopy from $f$ to $f\circ \phi$. 
By applying the same argument as above for each pass move, we show the equation we need. 
\end{proof}
\begin{rem}
We remark that the number of signed pass moves
of Theorem \ref{thm:signed-number} depends on a choice of 
an oriented spine of $F$. 
We choose an oriented spine $S$ in $F$. 
Then, by Theorem \ref{thm:signed-number}, 
$$
\mathcal{P}_S(\phi) 
= ( \text{ the number of } + \text{ pass moves }) - 
( \text{ the number of } - \text{ pass moves })
$$
is well-defined as a function of $\phi \in \mathcal{E}_f$. 
This function $\mathcal{P}_S$ satisfies 
$\mathcal{P}_S (\phi \circ \psi) = \mathcal{P}_{\psi(S)} (\phi) + 
\mathcal{P}_S (\psi)$ for $\phi, \psi \in \mathcal{E}_f$, 
since 
$$
\begin{aligned}
\mathcal{P}_S (\phi \circ \psi) =& 
\frac{1}{2} 
\left\{ \tilde{q}_f ((\phi \circ \psi)_*(\alpha_1 + \cdots + \alpha_{2g+k-1})) 
- \tilde{q}_f (\alpha_1 + \cdots + \alpha_{2g+k-1}) \right\}\\
=& 
\frac{1}{2} 
\left\{ \tilde{q}_f (\phi_* (\psi_*\alpha_1 + \cdots + \psi_*\alpha_{2g+k-1})) 
- \tilde{q}_f (\psi_*\alpha_1 + \cdots + \psi_*\alpha_{2g+k-1}) \right\} \\
&+ \frac{1}{2} 
\left\{ \tilde{q}_f (\psi_*(\alpha_1 + \cdots + \alpha_{2g+k-1})) 
- \tilde{q}_f (\alpha_1 + \cdots + \alpha_{2g+k-1}) \right\} \\
=& \mathcal{P}_{\psi(S)} (\phi) + \mathcal{P}_S (\psi). 
\end{aligned}
$$
\end{rem}


\begin{thebibliography}{99}
%
\bibitem{Kato} {\bf M. ~Kato},
{\it Topology of bands}, 
S\={u}rikaisekikenky\={u}sho K\={o}ky\={u}roku \ 243 (1975) 88--96. 
%
\bibitem{Nowik} {\bf T. ~Nowik}, 
{\it Automorphisms and embeddings of surfaces and quadruple points of
regular homotopies},
J. Differential Geometry \ 58 (2001) 421--455. 
%
\bibitem{Pinkall} {\bf U. ~Pinkall}, 
{\it Regular homotopy classes of immersed surfaces}, 
Topology \ 24 (1985) 421--434. 
%
\end{thebibliography}
\end{document}